\documentclass[reqno,12pt]{amsart}
\usepackage[a4paper,height=20cm,top=26mm,bottom=26mm,left=26mm,right=26mm]{geometry}

\usepackage{subfiles}
\usepackage{cleveref}

\usepackage{etex}
\usepackage{amsthm,amssymb}
\usepackage{mathrsfs}%
\usepackage{epic,eepic}
\usepackage{xypic}

\usepackage{here}
\usepackage{bm}
\usepackage[all]{xy}
\usepackage{tikz}
\usetikzlibrary{calc}

\numberwithin{equation}{section}

\newtheorem{thm}{Theorem}[section]

\newtheorem{lem}[thm]{Lemma}
\newtheorem{prop}[thm]{Proposition}

\theoremstyle{definition}
\newtheorem{definition}[thm]{Definition}

\theoremstyle{remark}

\crefname{thm}{Theorem}{Theorems}
\crefname{cor}{Corollary}{Corollaries}
\crefname{lem}{Lemma}{Lemmas}
\crefname{prop}{Proposition}{Propositions}
\crefname{definition}{Definition}{Definitions}
\crefname{example}{Example}{Examples}
\crefname{claim}{Claim}{Claims}
\crefname{conjecture}{Conjecture}{Conjectures}
\crefname{remark}{Remark}{Remarks}
\crefname{figure}{Figure}{Figures}
\crefname{section}{Section}{Sections}
\crefname{subsection}{Section}{Sections}

\crefname{introthm}{Theorem}{Theorems}
\crefname{introcor}{Corollary}{Corollaries}
\crefname{introconj}{Conjecture}{Conjectures}

\def\C{{\mathbb C}}
\def\Z{{\mathbb Z}}

\newcommand\mg{{\mathfrak{g}}}

\newcommand\cF{{\mathcal{F}}}
\newcommand\cM{{\mathcal{M}}}
\newcommand\cY{{\mathcal{Y}}}
\newcommand\cZ{{\mathcal{Z}}}

\newfont{\bg}{cmr9 scaled\magstep4}%BIGZERO
\newcommand{\bigzerol}{\smash{\lower1.0ex\hbox{\bg 0}}}

\newcommand\qline[2]{\draw[-,shorten >=2pt,shorten <=2pt] (#1) -- (#2);}% [thick];} %arrow
\tikzset{
  mid arrow/.style={postaction={decorate,decoration={
        markings,
        mark=at position .5 with {\arrow[#1]{stealth}}
      }}},
}

\newcommand{\ol}{\overline}

\newcommand{\wt}{\widetilde}

\begin{document}
\title[Invariants of Weyl group action and $q$-characters]
{Invariants of Weyl group action and $q$-characters of quantum affine algebras}% at roots of unity}

\author[Rei Inoue]{Rei Inoue}
\address{Rei Inoue, Department of Mathematics and Informatics,
   Faculty of Science, Chiba University,
   Chiba 263-8522, Japan.}
\email{reiiy@math.s.chiba-u.ac.jp}

\author[Takao Yamazaki]{Takao Yamazaki}
\address{Takao Yamazaki, Department of Mathematics,
Chuo University, Bunkyo-ku, Kasuga, 1-13-27
Tokyo, 112-8551, Japan}
\email{ytakao@math.chuo-u.ac.jp}

%\date{\today}
\date{July 20, 2022}

%%%%%%%%%%%%%%%%%%%%%%%%%%%

\begin{abstract}
Let $W$ be the Weyl group corresponding to a finite dimensional simple Lie algebra $\mg$ of rank $\ell$ and let $m>1$ be an integer. 
In \cite{I21}, by applying cluster mutations,
a $W$-action on $\mathcal{Y}_m$ was constructed.
Here $\mathcal{Y}_m$ is 
the rational function field on 
$cm\ell$ commuting variables,
where $c \in \{ 1, 2, 3 \}$ depends on $\mg$.
This was motivated by the $q$-character map $\chi_q$ of the category of finite dimensional representations of quantum affine algebra $U_q(\hat{\mg})$. We showed in \cite{I21} that when $q$ is a root of unity, %$q^{2 d'm} = 1$, 
$\mathrm{Im} \chi_q$ is a subring of the $W$-invariant subfield $\mathcal{Y}_m^W$ of $\mathcal{Y}_m$. In this paper, we give more detailed study on $\mathcal{Y}_m^W$; for each reflection $r_i \in W$ associated to the $i$th simple root, we describe the $r_i$-invariant subfield $\mathcal{Y}_m^{r_i}$ of $\mathcal{Y}_m$.  
\end{abstract}

\keywords{Weyl group, $q$-character, cluster algebras}

%\subjclass[2000]{}

\maketitle

%\tableofcontents

\section{Introduction}

Let $\mg$ be a finite dimensional simple Lie algebra of rank $\ell$, 
and fix a positive integer $m > 1$. 
Let $I := \{1,2,\ldots,\ell \}$ be the rank set of $\mg$.
In \cite{I21} we defined an action of the Weyl group $W$ on the rational function field $\mathcal{Y}_m$ generated by free variables $y_i(n) ~(i \in I, ~ \in d \Z/ m d' \Z)$. 
Here $d$ and $d'$ are rational numbers determined from the root system for $\mg$ (see \S \ref{subsec:Lie-alg} for the definition). This Weyl group action was originally defined by sequences of cluster mutations on the cluster seeds \cite{ILP16, IIO19, I21} associated to some periodic quivers,
and extended to that on $\mathcal{Y}_m$ in \cite{I21}.

The motivation to introduce $y_i(n)$ was the $q$-characters for finite dimensional representations of quantum non-twisted affine algebras $U_q(\hat{\mg})$ studied by Frenkel and Reshetikhin \cite{FR98,FR99}. 
The $q$-character $\chi_q$ is a ring homomorphism, 
$$
  \chi_q : \mathrm{Rep}\,U_q(\hat{\mg}) \to \mathbf{Y} := \Z[Y_{i,a_i}^{\pm 1}; i \in I, a_i \in \C^\times],
$$ 
from the Grothendieck ring $\mathrm{Rep}\,U_q(\hat{\mg})$ of 
the category  of finite dimensional representations of $ U_q(\hat{\mg})$ to the Laurent polynomial ring $\mathbf{Y}$ generated by commuting variables $Y_{i,a_i}$. 
For a generic $q$, $\mathrm{Rep}\,U_q(\hat{\mg})$ is parametrized by $a \in \C^\times/q^{d \Z}$, and the ring $\mathbf{Y}$ is stratified as $\mathbf{Y} = \otimes_{a \in \C^\times/q^{d \Z}} \mathbf{Y}_{a}$, where $\mathbf{Y}_{a} := \Z[Y_{i,a q^n}^{\pm 1};i \in I, n \in d\Z]$. 
The intersection of 
$\mathrm{Im} \chi_q$ and $\mathbf{Y}_{a}$ is known to be 
\begin{align}\label{eq:q-char}
\mathrm{Im}\chi_q \cap \mathbf{Y}_{a}
= \bigcap_{i \in I} \Z[Z_{i,a q^n}, Y_{j,a q^n}^{\pm 1};
j \in I \setminus \{ i \}, n \in d\Z],
\end{align}
where the $Z_{i,a q^n}$ are Laurent binomials in $\mathbf{Y}_{a}$.

When $q$ is a root of unity, $q^{2 d'm} = 1$, the above structure of the $q$-character map is basically preserved; we just put the condition $q^{2 d'm} = 1$ to \eqref{eq:q-char} \cite{FM01}.
We showed in \cite{I21} that, by identifying $\mathbf{Y}_a$ with $\Z[y_i(n)^{\pm 1};i \in I, n \in d \Z/d'm \Z]$, $\mathrm{Im}\chi_q \cap \mathbf{Y}_a$ is contained in the $W$-invariant subfield $\mathcal{Y}_m^W$ of $\mathcal{Y}_m$.

The aim of this paper is to study $\mathcal{Y}_m^W$ in more depth.
For $i \in I$, define a subfield $\mathcal{Z}^{(i)}_m$ of $\mathcal{Y}_m$ by $\mathcal{Z}^{(i)}_m := \C(z_i(n), y_j(n); j \in I \setminus \{ i \}, n \in d\Z/d'm\Z)$, where $z_i(n)$ are Laurent binomials in the $y_j(n)$ given by \eqref{psi-def}, corresponding to $Z_{i,a q^n}$ appearing in \eqref{eq:q-char}. Let $\alpha_i$ be the $i$th simple root, and 
$r_i \in W$ be the reflection associated to $\alpha_i$. 
Our main result is as follows.    
\begin{thm}[Theorem \ref{thm:simply-laced}, \ref{thm:non-simply-laced}]
For each $i \in I$ such that $\alpha_i$ is a shortest root, the $r_i$-invariant subfield $\mathcal{Y}_m^{r_i}$ of $\mathcal{Y}_m$ agrees with $\mathcal{Z}^{(i)}_m$.
For each $i \in I$ such that $\alpha_i$ is not a shortest root, $\mathcal{Y}_m^{r_i}$ agrees with an extension $\mathcal{Z}^{(i) \prime}_m$ of $\mathcal{Z}^{(i)}_m$ explicitly constructed in \eqref{eq:def-sZprime} below,
whose degree is either two or four
according to $d_i=2d$ or $d_i=3d$.
\end{thm}

In the case $\mg=A_1$, 
the theorem says that $\mathcal{Z}^{(1)}_m = \mathcal{Y}_m^W$.
For general $\mg$ 
it seems difficult to find a set of generators 
of $\mathcal{Y}_m^W$.
We leave this as an open problem.

\subsection*{Related topics}
The Weyl group action studied in this paper is related to cluster algebraic structure. We remark about some topics.

In \cite{ILP16}, a realization of the Weyl group for $\mg = A_\ell$ was defined as sequences of cluster mutations in triangular grid quivers on a cylinder with $m \ell$ vertices. It was shown that 
the affine geometric $R$-matrix of symmetric power representations for the quantum affine algebra $U'_q(A_\ell^{(1)})$ is obtained from the Weyl group realization. The quantization of the geometric $R$-matrix is also introduced by applying quantum cluster mutations.  
This cluster realization of Weyl groups is generalized to that for a symmetrizable Kac-Moody Lie algebra in \cite{IIO19}. When a Lie algebra $\mg$ is finite dimensional and $m$ is the Coxeter number of $\mg$, this cluster structure has an application in higher Teichm\"uller theory \`a la Fock and Goncharov \cite{FG03} as studied in \cite{GS16,IIO19,GS19}. This is also related to positive representations of $U_q(\mg)$ \cite{Ip16, SS16}.

On the other hand, for a finite dimensional Lie algebra $\mg$, the cluster structure of the $q$-characters for a finite dimensional representation of the affine quantum group $U_q(\hat{\mg})$ was studied by Hernandez and Leclerc \cite{HL16}, by introducing an infinite quiver. When $\mg$ has a simply laced Dynkin diagram, this quiver reduces to what was used in \cite{IIO19} by setting $m$-periodicity. 
The quivers used in \cite{I21} correspond to the periodic versions of \cite{HL16} for all $\mg$.

\subsection*{Contents of the paper}

This paper is organized as follows. In \S 2, after fixing basic notations in Lie algebras, we recall the Weyl group action on $\mathcal{Y}_m$ introduced in \cite{I21}.
In \S 3 and \S 4, we study the $W$-invariant subfield $\mathcal{Y}_m^W$ when $\mg$ has a simply laced Dynkin diagram and a non-simply laced Dynkin diagram respectively.

\subsection*{Acknowledgement}
RI is supported by JSPS KAKENHI Grant Number 19K03440. 
TY is supported by JSPS KAKENHI Grant Number 21K03153.

%%%%%%%%%%%%%%%%%%%%%%%%%%%%%%
\section{Weyl group action on $\mathcal{Y}_m$}
%%%%%%%%%%%%%%%%%%%%%%%%%%%%%%

\subsection{Lie algebras and Weyl groups}\label{subsec:Lie-alg}
First we recall notations related to Lie algebras.
Let $\mg$ be a finite-dimensional simple Lie algebra of rank $\ell$ over $\C$. Denote its rank set by $I = \{1,2,\ldots,\ell \}$. 
%Let $\mathfrak{h}$ be the Cartan subalgebra of $\mg$.
For $i \in I$, we write $\alpha_i$ for the $i$th simple root.
% and $\omega_i$ for the $i$th fundamental weight.
The Cartan matrix  
$(C_{ij})_{i,j \in I}$
is given by 
$$
C_{ij} = 2 \, \frac{(\alpha_i \,,\,\alpha_j)}{(\alpha_i\,,\,\alpha_i)},
$$
where $( ~~ ,~~ )$ is the inner product. See Figure \ref{Dynkin-diagrams} for the convention of the Dynkin diagrams in this paper.
We  define
\begin{align}\label{eq:def-d}
  d_i = \frac{1}{2} ~( \alpha_i \,,\, \alpha_i),
\quad d= \min\{ d_i ; ~i \in I\}, 
\quad d'= \max\{ d_i ; ~i \in I \},
\end{align}
which are explicitly given by the following table:
\begin{align}\label{table:dd}
\begin{tabular}{cll}
$A_\ell, ~D_\ell, ~E_\ell $ :
& $d_i=1 ~(i=1, \dots, \ell)$,~~~
&$d=d'=1$ 
\\
$B_\ell $ :
& $d_i=1 ~(i=1, \dots, \ell-1), ~d_\ell=\frac{1}{2}$,~~~
&$d=\frac{1}{2}, ~d'=1$
\\
$C_\ell $ :
& $d_i=1 ~(i=1, \dots, \ell-1), ~d_\ell=2$,~~~
&$d=1, ~d'=2$
\\
$F_4 $ :
& $d_1=d_2=1, ~d_3=d_4=\frac{1}{2}$,~~~
&$d=\frac{1}{2}, ~d'=1$ 
\\
$G_2 $ :
& $d_1=1, d_2=3$,~~~
&$d=1, ~d'=3$
\end{tabular}
\end{align}

The Weyl group $W$ associated with $\mg$ 
admits the following presentation:
$$
W=\langle r_i; ~i \in I \mid (r_i r_j)^{m_{ij}}=1; ~i, j\in I \rangle.
$$
Here 
$r_i \in W$ is the reflection associated to $\alpha_i$, and
$(m_{ij})_{i,j \in I}$ is a symmetric matrix 
given by $m_{ii}=1$ for all $i$
and by the following table for $i \neq j$:
\[
\begin{tabular}{rccccc}
$C_{ij}C_{ji}:$ & $0$ & $1$ & $2$ & $3$ %& $\geq 4$ 
\\
$m_{ij}:$         & $2$ & $3$ & $4$ & $6.$ %& $\infty$
\end{tabular}
\]

\begin{figure}[ht]
\begin{tikzpicture}
\begin{scope}[>=latex]
\draw (0,9) node{$A_\ell:$};
\draw (1,9) circle(2pt) coordinate(A1) node[above]{$1$};
\draw (2.5,9) circle(2pt) coordinate(A2) node[above]{$2$};
\draw (4,9) circle(2pt) coordinate(A3) node[above]{$3$};
\draw (8,9) circle(2pt) coordinate(A4) node[above]{$\ell-1$};
\draw (9.5,9) circle(2pt) coordinate(A5) node[above]{$\ell$};
\qline{A2}{A1}
\qline{A3}{A2}
\draw[-,shorten >=2pt,shorten <=2pt] (5.5,9) -- (A3);
\draw (6,9) node{$\dots$};
\draw[-,shorten >=2pt,shorten <=2pt] (A4) -- (6.5,9);
\qline{A5}{A4}

\draw (0,7.5) node{$B_\ell:$};
\draw (1,7.5) circle(2pt) coordinate(B1) node[above]{$1$};
\draw (2.5,7.5) circle(2pt) coordinate(B2) node[above]{$2$};
\draw (4,7.5) circle(2pt) coordinate(B3) node[above]{$3$};
\draw (8,7.5) circle(2pt) coordinate(B4) node[above]{$\ell-1$};
\draw (9.5,7.5) circle(2pt) coordinate(B5) node[above]{$\ell$};
\qline{B2}{B1}
\qline{B3}{B2}
\draw[-,shorten >=2pt,shorten <=2pt] (5.5,7.5) -- (B3);% [thick];
\draw (6,7.5) node{$\dots$};
\draw[-,shorten >=2pt,shorten <=2pt] (B4) -- (6.5,7.5);% [thick];
%\qline{B5}{B4}
\draw[-,shorten >=2pt,shorten <=2pt] (8,7.475) -- (9.5,7.475);% [thick];
\draw[-,shorten >=2pt,shorten <=2pt] (8,7.525) -- (9.5,7.525);% [thick];
\draw (8.75,7.5) node{$>$};

\draw (0,6) node{$C_\ell:$};
\draw (1,6) circle(2pt) coordinate(C1) node[above]{$1$};
\draw (2.5,6) circle(2pt) coordinate(C2) node[above]{$2$};
\draw (4,6) circle(2pt) coordinate(C3) node[above]{$3$};
\draw (8,6) circle(2pt) coordinate(C4) node[above]{$\ell-1$};
\draw (9.5,6) circle(2pt) coordinate(C5) node[above]{$\ell$};
\qline{C2}{C1}
\qline{C3}{C2}
\draw[-,shorten >=2pt,shorten <=2pt] (5.5,6) -- (C3);% [thick];
\draw (6,6) node{$\dots$};
\draw[-,shorten >=2pt,shorten <=2pt] (C4) -- (6.5,6);% [thick];
\draw[-,shorten >=2pt,shorten <=2pt] (8,5.975) -- (9.5,5.975);% [thick];
\draw[-,shorten >=2pt,shorten <=2pt] (8,6.025) -- (9.5,6.025);% [thick];
\draw (8.75,6) node{$<$};

\draw (-0.4,4.5) node{$D_\ell ~(\ell \geq 4):$};
%\draw (-0.2,4.5) node{$D_{\ell \geq 4}:$};
\draw (1,4.5) circle(2pt) coordinate(D1) node[above]{$1$};
\draw (2.5,4.5) circle(2pt) coordinate(D2) node[above]{$2$};
\draw (4,4.5) circle(2pt) coordinate(D3) node[above]{$3$};
\draw (8,4.5) circle(2pt) coordinate(D4) node[above]{$\ell-2$};
\draw (9.5,4.5) circle(2pt) coordinate(D5) node[above]{$\ell-1$};
\draw (8,3) circle(2pt) coordinate(D6) node[right]{$\ell$};
\qline{D2}{D1}
\qline{D3}{D2}
\draw[-,shorten >=2pt,shorten <=2pt] (5.5,4.5) -- (D3);% [thick];
\draw (6,4.5) node{$\dots$};
\draw[-,shorten >=2pt,shorten <=2pt] (D4) -- (6.5,4.5);% [thick];
\qline{D5}{D4}
\qline{D6}{D4}

\draw (-0.7,2) node{$E_\ell~(\ell=6,7,8):$}; 
%\draw (-0.4,2) node{$E_{\ell=6,7,8}:$}; 
\draw (1,2) circle(2pt) coordinate(E1) node[above]{$1$};
\draw (2.5,2) circle(2pt) coordinate(E2) node[above]{$2$};
\draw (4,2) circle(2pt) coordinate(E3) node[above]{$3$};
\draw (5.5,2) circle(2pt) coordinate(E4) node[above]{$5$};
\draw (8,2) circle(2pt) coordinate(E5) node[above]{$\ell-1$};
\draw (9.5,2) circle(2pt) coordinate(E6) node[above]{$\ell$};
\draw (4,0.5) circle(2pt) coordinate(E7) node[right]{$4$};
\qline{E2}{E1}
\qline{E3}{E2}
\qline{E4}{E3}
%\draw[-,shorten >=2pt,shorten <=2pt] (5.5,2.5) -- (E3);% [thick];
\draw (6,2) node{$\dots$};
\draw[-,shorten >=2pt,shorten <=2pt] (E5) -- (6.5,2);% [thick];
\qline{E6}{E5}
\qline{E7}{E3}

\draw (0,-0.5) node{$F_4:$};
\draw (1,-0.5) circle(2pt) coordinate(F1) node[above]{$1$};
\draw (2.5,-0.5) circle(2pt) coordinate(F2) node[above]{$2$};
\draw (4,-0.5) circle(2pt) coordinate(F3) node[above]{$3$};
\draw (5.5,-0.5) circle(2pt) coordinate(F4) node[above]{$4$};
\qline{F2}{F1}
\qline{F4}{F3}
\draw[-,shorten >=2pt,shorten <=2pt] (2.5,-0.475) -- (4,-0.475);% [thick];
\draw[-,shorten >=2pt,shorten <=2pt] (2.5,-0.525) -- (4,-0.525);% [thick];
\draw (3.25,-0.5) node{$>$};

\draw (7,-0.5) node{$G_2:$};
\draw (8,-0.5) circle(2pt) coordinate(G1) node[above]{$1$};
\draw (9.5,-0.5) circle(2pt) coordinate(G2) node[above]{$2$};
\qline{G2}{G1}
\draw[-,shorten >=2pt,shorten <=2pt] (8,-0.45) -- (9.5,-0.45); %[thick];
\draw[-,shorten >=2pt,shorten <=2pt] (8,-0.55) -- (9.5,-0.55); % [thick];
\draw (8.75,-0.5) node{$<$};

\end{scope}
\end{tikzpicture}
\caption{Dynkin diagrams for $\mg$}
\label{Dynkin-diagrams}
\end{figure}

\subsection{Weyl group action}

We fix an integer $m > 1$,
and let
\begin{equation}\label{eq:def-Ym-general}
\mathcal{Y}_m
:= \C(y_i(n);  i \in I, n \in d\Z/d'm\Z)
\end{equation}
be the rational function field 
on the commuting variables 
$y_i(n), ~(i,n) \in I \times d \Z/d'm\Z$.
We define elements of $\mathcal{Y}_m$ 
for $(i,n) \in I \times d \Z/d'm\Z$ as follows:
\begin{align}
\label{eq:a-y-g}
&F_i(n)
=
\begin{cases}
y_{i-1}(n+\frac{1}{2}) y_{i+1}(n) & 
(\mg, i)=(B_\ell, \ell), \,(F_4, 3);
\\
y_{i-1}(n+1) y_{i+1}(n) y_{i+1}(n+\frac{1}{2}) & 
(\mg, i)=(B_\ell, \ell-1), \, (F_4, 2);
\\
y_{\ell-1}(n+1) y_{\ell-1}(n+2) & 
(\mg, i)=(C_\ell, \ell);
\\
y_1(n+1)y_1(n+2)y_1(n+3) & 
(\mg, i)=(G_2, 2); 
\\
\displaystyle{\prod_{j: j < i, C_{ij}\neq 0} y_j(n+d_j) \prod_{j:j > i,C_{ij}\neq 0} y_j(n)} & \text{otherwise; }
\end{cases}
\\
\label{a-y}
&X_i(n) = \frac{F_i(n)}{y_i(n) y_{i}(n+d_i)};
\\
\label{eq:def-fsmall}
&P_i(n) := 1 + \sum_{k=0}^{\frac{d'm}{d_i}-2} X_i(n) X_i(n-d_i) \cdots X_i(n-d_i k).
\end{align}

\begin{thm}[Theorem 4.2, \cite{I21}]
\label{thm:WonY-g}
There is an action of $W$ on $\mathcal{Y}_m$
characterized by 
\begin{align}\label{Weyl-y}
r_i(y_j(n)) = 
\begin{cases}
\displaystyle{\frac{P_i(n-2d_i)}{P_i(n-d_i)} y_i(n) X_i(n-d_i)} & j = i,
\\[1mm]
y_j(n) & j \neq i,
\end{cases}
\end{align}
where $i, j \in I$ and $n \in d\Z/d'm\Z$.
\end{thm}

We are going to discuss 
the $W$-invariant subfield $\mathcal{Y}_m^{W}$ of $\mathcal{Y}_m$.
For $i \in I$, 
we define a subfield $\mathcal{Z}^{(i)}_m$ of $\mathcal{Y}_m$ by 
\begin{align}\label{eq:def-Zim}
\mathcal{Z}^{(i)}_m 
:= \C(z_i(n), y_j(n); j \in I \setminus \{ i \}, n \in d\Z/d'm\Z),
\end{align}
where we put
\begin{align}\label{psi-def}
  z_i(n) :=  
y_i(n) + \frac{F_i(n)}{y_i(n+d_i)}
=y_i(n)(1 + X_i(n)).
\end{align}

\begin{thm}[Corollary of Proposition 4.13, \cite{I21}]
\label{thm:I2-2}
We have $\mathcal{Z}_m^{(i)} \subset \mathcal{Y}_m^{r_i}$
for any $i \in I$, 
where $\mathcal{Y}_m^{r_i}$ is the $r_i$-invariant subfield of $\mathcal{Y}_m$. 
We thus have $\bigcap_{i \in I} \mathcal{Z}^{(i)}_m \subset \mathcal{Y}_m^W$.
\end{thm}

%%%%%%%%%%%%%%%%%%%
\section{Invariant subfield $\mathcal{Y}_m^W$: simply-laced cases}
\label{sect:simplylaced}
%%%%%%%%%%%%%%%%%%%

\subsection{Main theorem and first reduction}
The goal of this section is the following:

\begin{thm}\label{thm:simply-laced}
Suppose that $\mg$ has a simply-laced Dynkin diagram
(that is, $\mg= A_\ell, ~D_\ell$ or $E_\ell$). 
Then we have $\mathcal{Y}_m^{r_i} = \mathcal{Z}^{(i)}_m$
for any $i \in I$.
%where $\mathcal{Y}_m^{r_i}$ denotes
%the $r_i$-invariant subfield of $\mathcal{Y}_m$.
Consequently, we have
$\mathcal{Y}_m^W = \bigcap_{i \in I} \mathcal{Z}^{(i)}_m$.
\end{thm}

In the rest of this section,
we keep a running assumption that
$\mg$ is associated to a simply-laced Dynkin diagram,
and we fix $i \in I$ and $m>1$.
Recall that we have then $d=d'=d_i=1$ for all $i$,
and hence \eqref{eq:a-y-g} reduces to 
\begin{equation}\label{eq:a-y-g-sl}
F_i(n) = \prod_{j: j < i, C_{ij}\neq 0} y_j(n+1) \prod_{j:j > i,C_{ij}\neq 0} y_j(n).
\end{equation}

We define three subfields of $\mathcal{Y}_m$ as follows
(see \eqref{eq:a-y-g}, \eqref{psi-def}):
\begin{equation}
\label{eq:def-cF}
\begin{aligned}
&\cF := \C(F_i(n); ~n \in \Z/m\Z), \\
&\cY_\cF := \cF(y_i(n); ~n \in \Z/m\Z), \\
&\cZ_\cF := \cF(z_i(n); ~n \in \Z/m\Z).
\end{aligned}
\end{equation}
Observe that we have
\begin{equation}\label{eq:Xi-cYcF}
X_i(n), ~P_i(n) \in \cY_\cF
\end{equation}
for all $n \in \Z/m\Z$ by \eqref{a-y} and \eqref{eq:def-fsmall}.

\begin{lem}\label{lem:ri-Yf}
The restriction of $r_i$ to $\cZ_\cF$ is the identity,
and
we have $r_i(\cY_\cF) \subset \cY_\cF$.
\end{lem}
\begin{proof}
We have
$r_i(F_i(n))=F_i(n)$ for any $n \in \Z/m\Z$
by \eqref{eq:a-y-g-sl}
and by the second case of \eqref{Weyl-y}.
Hence
the first statement follows from Theorem \ref{thm:I2-2}.
It remains to prove $r_i(y_i(n)) \in \cY_\cF$,
but this is immediate from 
\eqref{eq:def-fsmall}, \eqref{Weyl-y} and \eqref{eq:Xi-cYcF}.
\end{proof}

We summarize the relations of the fields in a diagram: 
\begin{equation}\label{eq:diag-fld}
\xymatrix{
\cZ_m^{(i)} \ar@{}[r]|*{\subset} \ar@{}[d]|{\bigcup} &
\cY_m^{r_i} \ar@{}[r]|*{\subset} \ar@{}[d]|{\bigcup} &
\cY_m  \ar@{}[d]|{\bigcup} 
\\
\cZ_\cF \ar@{}[r]|*{\subset} &
\cY_\cF^{r_i} \ar@{}[r]|*{\subset}  &
\cY_\cF,
}
\end{equation}
where $\mathcal{Y}_\cF^{r_i}$ is
the $r_i$-invariant subfield of $\cY_\cF$.
Here we make a first reduction:

\begin{lem}\label{lem:reduction}
An equality
\begin{equation}\label{prop:mid-step}
[\cY_\cF : \cZ_\cF]=2
\end{equation}
implies Theorem \ref{thm:simply-laced}.
\end{lem}
\begin{proof}
We have 
$[\cY_m  : \cZ_m^{(i)}] \le [\cY_\cF : \cZ_\cF]$
since
$\cY_m$ is the composition field of $\cY_\cF$ and $\cZ_m^{(i)}$
by definition.
On the other hand,
we have 
$[\cY_\cF : \cY_\cF^{r_i}]=[\cY_m : \cY_m^{r_i}]=2$
because $r_i$ is of order two.
Therefore
\eqref{prop:mid-step} implies $\cZ_\cF=\cY_\cF^{r_i}$
and hence $\cZ_m^{(i)}=\cY_m^{r_i}$.
\end{proof}

\subsection{The proof}
In order to prove \eqref{prop:mid-step}, we introduce the Laurent polynomial ring
\[
%\mathcal{Y} := \C[y_i(n)^{\pm 1}; i \in I, n \in \Z]
\mathcal{Y}_\infty := \cF[\wt{y}_i(n)^{\pm 1}; n \in \Z]
\]
on the set of commuting variables 
$\wt{y}_i(n)$ on $n \in \Z$ over $\cF$.
We also introduce its $\cF$-subalgebra
\begin{equation}\label{psi-def3}
\mathcal{Z}_\infty:= \cF[\wt{z}_i(n);~ n \in \Z] \subset \cY_\infty,
\qquad
  \wt{z}_i(n) = \wt{y}_i(n) + \frac{F_i(n \bmod m)}{\wt{y}_i(n+1)} .
\end{equation}

\begin{lem}\label{lem:alg-indep}
The set $\{ \wt{z}_i(n) ;~ n \in \Z \}$ 
is algebraically independent over $\cF$. 
In particular, 
$\cZ_\infty$ is a polynomial ring over $\cF$.
\end{lem}

\begin{proof}
The set
$\{ \wt{y}_i(n) ; n \in \Z \}$ 
is algebraically independent over $\cF$
by definition.
On the other hand, 
it follows from \eqref{psi-def3} that 
for any $N>0$ the two sets
\[ \{ \wt{z}_i(n) ; -N \le n \le N \} \cup \{ \wt{y}_i(0) \}
\qquad \text{and} \qquad
   \{ \wt{y}_i(n) ; -N \le n \le N+1 \}
\]
generate (over $\cF$) the same subfield in
the fraction field of $\mathcal{Y}_\infty$.
Since the two sets have the same cardinality,
the first is algebraically independent over $\cF$ as well.
We are done.
\end{proof}

\begin{definition}\label{lem:shift-op}
\begin{enumerate}
\item 
Let
$\tau_m : \cY_m \to \cY_m$ 
be a $\C$-algebra automorphism
characterized by
$\tau_m y_j(n) = y_j(n+1)$
for any $(j, n) \in I \times \Z/m\Z$.
We have (see \eqref{eq:def-cF})
\[ \tau_m(\cZ_m) = \cZ_m,
\qquad
  \tau_m(\cF) = \cF,
\qquad
  \tau_m(\cY_\cF) = \cY_\cF,
\qquad
  \tau_m(\cZ_\cF) = \cZ_\cF.
\]
We denote by $\tau_\cF : \cF \to \cF$
the restriction of $\tau_m$.
\item
Let
$\tau_\infty : \cY_\infty \to \cY_\infty$
be a $\C$-algebra automorphism
characterized by
$\tau_\infty \wt{y}_i(n) = \wt{y}_i(n+1)$
for any $n \in  \Z$
and $\tau_\infty|_{\cF}=\tau_\cF$.
We have $\tau_\infty(\cZ_\infty) = \cZ_\infty$.
\item
Let
$\pi : \mathcal{Y}_\infty \to \mathcal{Y}_\cF$
be a $\cF$-algebra homomorphism
characterized by
$\pi \wt{y}_i(n)=y_i(n \bmod m)$
for all $n \in \Z$.
We have $\pi(\cZ_\infty) \subset \cZ_\cF$ and
$\tau_m \circ \pi = \pi \circ \tau_\infty$.
\end{enumerate}
\end{definition}

We define polynomials $\wt{A}^{(k)}, \wt{C}^{(k)}$ in $\cZ_\infty$ as follows. First we define $\wt{C}^{(k)}$ for $k \in \Z_{>0}$ by
\begin{align}
  \label{eq:g1-C12}
  &\wt{C}^{(1)} = 1, ~~ \wt{C}^{(2)} = \wt{z}_i(2),
  \\
  \label{eq:g1-Ck}
  &\wt{C}^{(k)} = \wt{z}_i(k) \wt{C}^{(k-1)} - F_i(k-1 \bmod m)\wt{C}^{(k-2)} \quad (k \ge 3).
\end{align} 
Next, for $k \geq 2$ we define $\wt{A}^{(k)}$ as
\begin{align}
\label{eq:g1-A}
  &\wt{A}^{(k)} = \wt{z}_i(1) \wt{C}^{(k)} - F_i(k \bmod m) \wt{C}^{(k-1)} - \tau_{\infty} (F_i(k \bmod m) C^{(k-1)}). 
\end{align}

We define elements of $\cY_\infty$ 
for $k \geq 2$ and $n \in \Z$ by 
\begin{equation}\label{eq:g1-def-small-f}
\wt{D}_n^{(k)} = 
1 + \sum_{p=0}^{k-2} \wt{X}_i(n)\wt{X}_i(n-1)\cdots \wt{X}_i(n-p),
\quad
\wt{X}_i(n) = \frac{F_i(n \bmod m)}{\wt{y}_i(n) \wt{y}_i(n+1)}.
\end{equation} 
Note that it is satisfied that
\begin{align}\label{eq:g1-f}
& \wt{D}_n^{(k)} = 1 + \wt{X}_i(n) \wt{D}_{n-1}^{(k-1)} 
= \wt{D}_{n}^{(k-1)} + \wt{X}_i(n) \wt{X}_i(n-1) \cdots \wt{X}_i(n-k+2),
\\
\label{eq:g1-f2}
&\wt{z}_i(n) = (1+ \wt{X}_i(n)) \wt{y}_i(n).
\end{align}

\begin{lem}\label{lem:g1-C-f}
It is satisfied that 
$\wt{C}^{(k)} 
= \wt{D}_k^{(k)} \wt{y}_i(2) \wt{y}_i(3)\cdots \wt{y}_i(k)$ 
in $\cY_\infty$. 
\end{lem}

\begin{proof}
Write $G_k$ for the r.h.s. in the statement.
We prove $\wt{C}^{(k)}=G_k$ by induction on $k$.
When $k=2$, we have
$$
  G_2=
\wt{D}_2^{(2)} \wt{y}_i(2) = (1+\wt{X}_i(2)) \wt{y}_i(2)
\overset{\eqref{eq:g1-f2}}{=} \wt{z}_i(2) = \wt{C}^{(2)}.
$$
When $k=3$, we have %the r.h.s. of the claim is written as
\begin{align*}
G_3&=\wt{D}_3^{(3)} \wt{y}_i(2) \wt{y}_i(3)
  = (1+\wt{X}_i(3)+ \wt{X}_i(3)\wt{X}_i(2))\wt{y}_i(2) \wt{y}_i(3),\\
  \wt{C}^{(3)} &= \wt{z}_i(3) \wt{z}_i(2) - F_i(2 \bmod m) \quad (\text{from \eqref{eq:g1-C12} and \eqref{eq:g1-Ck}})  
  \\
  &= (1+\wt{X}_i(3))\wt{y}_i(2) (1+\wt{X}_i(2))\wt{y}_i(2) - F_i(2 \bmod m) \quad (\text{from \eqref{eq:g1-f2}})
  \\
  &= (1+\wt{X}_i(3)+ \wt{X}_i(3) \wt{X}_i(2))\wt{y}_i(2) \wt{y}_i(2) + \wt{X}_i(2)\wt{y}_i(2) \wt{y}_i(3) - F_i(2 \bmod m). 
\end{align*}
The last two terms vanish due to the second formula of \eqref{eq:g1-def-small-f}, and the claim is shown.
For $k \geq 4$ we prove that $G_k$ satisfies the same recurrence formula \eqref{eq:g1-Ck} as $\wt{C}^{(k)}$. 
By using the first formula of \eqref{eq:g1-f} twice, we obtain 
\[
(1+\wt{X}_i(k))\wt{D}^{(k-1)}_{k-1}
=
\wt{D}^{(k)}_{k} + \wt{X}_i(k-1)\wt{D}^{(k-2)}_{k-2}.
\]
It then follows from \eqref{eq:g1-f2} that
\begin{align*}
\wt{z}_i(k) G_{k-1}
&= 
(1+\wt{X}_i(k))\wt{y}_i(k) \cdot
\wt{D}^{(k-1)}_{k-1}
\wt{y}_i(2) \wt{y}_i(3) \cdots \wt{y}_i(k-1)
\\
&= 
\left(\wt{D}^{(k)}_{k} + \wt{X}_i(k-1)\wt{D}^{(k-2)}_{k-2} \right)
\wt{y}_i(2) \wt{y}_i(3) \cdots \wt{y}_i(k).
\end{align*}
On the other hand, by \eqref{eq:g1-def-small-f} we get
\[  F_i(k-1 \bmod m)G_{k-2}
= \wt{y}_i(k-1) \wt{y}_i(k) \wt{X}_i(k-1) \cdot
\wt{D}^{(k-2)}_{k-2}
\wt{y}_i(2) \wt{y}_i(3) \cdots \wt{y}_i(k-2).
\]
Combined, we arrive at the desired formula
\begin{align*}
\wt{z}_i(k) G_{k-1} - F_i(k-1 \bmod m) G_{k-2}
&= \wt{y}_i(2) \wt{y}_i(3) \cdots \wt{y}_i(k) \wt{D}_{k}^{(k)} = G_k,
\end{align*}    
and the claim follows.
\end{proof}

For $n \in \Z$ and $2 \leq k \leq m$, we define
\[ 
A^{(k)}:=\pi(\wt{A}^{(k)}) \in \cZ_\cF,~
%\overline{B}^{(k)}:=\pi(B_n^{(k)}),~
\quad
C^{(k)}:=\pi(\wt{C}^{(k)}) \in \cZ_\cF,
\quad
D_n^{(k)}:=\pi(\wt{D}_n^{(k)}) \in \cY_\cF,
\] 
where
$\pi : \cY_\infty \to \cY_\cF$
is from Definition \ref{lem:shift-op} (3).
We also define two elements by
\[
\boldsymbol{y}_i := \prod_{p \in \Z/m\Z} y_i(p) \in \cY_\cF,
\qquad
\boldsymbol{F}_i:= \prod_{n \in \Z/m\Z} F_i(n) \in \cF.
\]
Notice that by \eqref{a-y} and \eqref{eq:def-fsmall} we have
\begin{align}
\pi(\wt{X}_i(n))=X_i(n \bmod m).
\end{align}

\begin{thm}\label{thm:g1-delta}
\begin{enumerate}
\item 
We have $\displaystyle{A^{(m)}= \boldsymbol{y}_i + \frac{\boldsymbol{F}_i}{\boldsymbol{y}_i}}$ in $\cY_\cF$. In particular, this element is invariant under $\tau_m$. 
\item 
We have $2 (y_i(1) C^{(m)}- F_i(m) C^{(m-1)})- A^{(m)}=\displaystyle{\boldsymbol{y}_i - \frac{\boldsymbol{F}_i}{\boldsymbol{y}_i}}$ in $\cY_\cF$.
In particular, this element is invariant under $\tau_m$,
which we denote by $\delta$.
\item 
We have 
$\cY_\cF=\cZ_\cF(\delta)$,~
$\delta^2 
=(A^{(m)})^2-4 \boldsymbol{F}_i \in \cZ_\cF$
and 
$r_i(\delta)=-\delta$.
\end{enumerate}
\end{thm}

\begin{proof}
(1)
By using the two formulas in \eqref{eq:g1-f}, we get
\[
D^{(m)}_{m}+X_i(1) D^{(m)}_{m}
=(1+X_i(m)D^{(m-1)}_{m-1}) + 
X_i(1)( D^{(m-1)}_m + X_i(2) \cdots X_i(m)).
\]
We then deduce from Lemma \ref{lem:g1-C-f} and \eqref{psi-def}
\begin{align*}
  z_i(1) C^{(m)} 
  &= (1+X_i(1)) \, D_m^{(m)} \boldsymbol{y}_i
  \\
  &= \left( 1 + X_i(m) D_{m-1}^{(m-1)} + X_i(1) D_{m}^{(m-1)} + X_i(1) X_i(2) \cdots X_i(m) \right) \boldsymbol{y}_i.
\end{align*}
On the other hand, 
by Lemma \ref{lem:g1-C-f} and \eqref{a-y}
we have
\begin{align*}
C^{(m-1)}
&= D_{m-1}^{(m-1)} \frac{\boldsymbol{y}_i}{y_i(1)y_i(m)}
= \frac{X_i(m)}{F_i(m)} D_{m-1}^{(m-1)}  \boldsymbol{y}_i,
\\
\tau_m(C^{(m-1)})
&= \frac{X_i(1)}{F_i(1)} D_{m}^{(m-1)}  \boldsymbol{y}_i.
\end{align*}
Thus we obtain from \eqref{eq:g1-A}
\begin{align*}
A^{(m)} 
&=z_i(1) C^{(m)} - F_i(m) C^{(m-1)} - F_i(1) \tau_m(C^{(m-1)})\\
&= (1 + X_i(1) X_i(2) \cdots X_i(m))\boldsymbol{y}_i  
= \boldsymbol{y}_i + \frac{\boldsymbol{F}_i}{\boldsymbol{y}_i},
\end{align*}
where we used \eqref{a-y} again.
This is obviously invariant under $\tau_m$.

(2) From Lemma \ref{lem:g1-C-f}, 
\eqref{eq:g1-def-small-f} and the first formula of \eqref{eq:g1-f}, 
we have in $\cY_\infty$ 
\begin{align*}
&\wt{y}_i(1) \wt{C}^{(m)}
= \wt{y}_i(1) \wt{y}_i(2) \cdots \wt{y}_i(m) \wt{D}^{(m)}_m 
= \wt{y}_i(1) \wt{y}_i(2) \cdots \wt{y}_i(m) 
(1+\wt{X}_i(m)\wt{D}^{(m-1)}_{m-1}),
\\
&F_i(m \bmod m)\wt{C}^{(m-1)}
= \wt{y}_i(1) \wt{y}_i(2) \cdots \wt{y}_i(m) 
\wt{X}_i(m)\wt{D}^{(m-1)}_{m-1}.
\end{align*}  
Combined with (1),
we obtain $2 (y_i(1) C^{(m)} - F_i(m) C^{(m-1)}) - A^{(m)}  = 2 \boldsymbol{y}_i - (\boldsymbol{y}_i + \frac{\boldsymbol{F}_i}{\boldsymbol{y}_i})$ in $\mathcal{Y}_\cF$.
This is again invariant under $\tau_m$.

(3)
We have 
$y_i(1) =F_i(m)C^{(m-1)}+(\delta+A^{(m)})/2 \in \cZ_{\cF}(\delta)$ 
by (2).
Since $\tau_m(\delta)=\delta$,
iterated application of $\tau_m$ yields
$y_i(n) \in \cZ_\cF(\delta)$ for any $n \in \Z/m\Z$,
showing the first statement.
The second one follows from from (1) and (2),
and the last one is a consequence of
\eqref{a-y}, \eqref{Weyl-y} and Lemma \ref{lem:ri-Yf}.
\end{proof}

\begin{proof}[Proof of Theorem \ref{thm:simply-laced}]
Theorem \ref{thm:g1-delta} (3) shows \eqref{prop:mid-step},
hence Lemma \ref{lem:reduction} completes the proof.
\end{proof}

%%%%%%%%%%%%%%%%%%%%%
\subsection{Appendix: expressions of $\wt{C}^{(k)}$ in $\mathcal{Z}_\infty$ and $A^{(m)}$ in $\mathcal{Z}_{\mathcal{F}}$}
%%%%%%%%%%%%%%%%%%%%%

The polynomials $\wt{C}^{(k)}$ and  $A^{(m)}$ have simple expressions in $\mathcal{Y}_\infty$ and $\mathcal{Y}_m$ respectively, 
as Lemma \ref{lem:g1-C-f} and Theorem \ref{thm:g1-delta} show.
However they are not expressed in terms of the generators of
$\mathcal{Z}_\infty$ and $\mathcal{Z}_m$.
In this subsection we present such expressions.
The results in this subsection will not be used in the sequel.

To describe $\wt{C}^{(k)}$ \eqref{eq:g1-Ck}, we introduce notations: 
\begin{align}\label{eq:A_1-L}
  &\cM_p^{(k)} = 
\bigl\{ \sigma \subset \{2,3,\ldots,k-1 \};
~|\sigma|=p, ~j \not= j'+1 ~\text{for any}~j, j' \in \sigma \},
  \\
\label{eq:A_1-L2}
&M^{(k)}_p = \sum_{\sigma \in \cM^{(k)}_p} \prod_{j \in \sigma} F_i(j \bmod m) \prod_{j' \in \overline{\sigma}} \wt{z}_i(j') \in \cZ_\infty
\end{align}
for $p=0, 1, 2,\ldots,[\frac{k-1}{2}]$,
where
$\overline{\sigma} := 
\{ j \in \{2,3,\ldots,k\} ;~ j, j-1 \not\in \sigma \}$.
We regard $\cZ_\infty$ as a graded $\cF$-algebra
by defining the degree of $\wt{z}_i(n)$ to be one for any $n \in \Z$
and those of any elements of $\cF$ to be zero
(see Lemma \ref{lem:alg-indep}).
Then $M^{(k)}_p$ is homogeneous of degree $k-1-2p$. 

\begin{prop}\label{prop:g1-ABC-1}
For $k \geq 2$, we have 
\begin{align}\label{eq:g1-C-z}
  \wt{C}^{(k)} = \sum_{p=0}^{[\frac{k-1}{2}]} (-1)^p M^{(k)}_p
\quad \text{in}~ \cZ_\infty.
\end{align}
\end{prop}

\begin{proof}
It is immediate from the definition that 
\begin{align}\label{eq:easyfromdef}
M_0^{(k)} = \prod_{p=2}^k \wt{z}_i(p),
\qquad 
M_1^{(3)}=F_i(2).
\end{align}
We now proceed by induction on $k$. 
It follows from \eqref{eq:g1-C12}, \eqref{eq:g1-Ck} 
and \eqref{eq:easyfromdef} that
\begin{align*}
&\wt{C}^{(2)}=\wt{z}_i(2)= M_0^{(2)},
\\
&\wt{C}^{(3)} = \wt{z}_i(3) \wt{C}^{(2)} - F_i(2)\wt{C}^{(1)} = \wt{z}_i(2)\wt{z}_i(3) -F_i(2) = M_0^{(3)} - M_1^{(3)},
\end{align*}
proving the cases $k=2, 3$.
For $k \ge 4$, by inductive hypothesis and \eqref{eq:g1-Ck} we have
\begin{align}\label{eq:ind-CandM}
%\wt{z}_i(k)  &\wt{C}^{(k-1)} - F_i(k-1 \bmod m) \wt{C}^{(k-2)}  \\ & 
\wt{C}^{(k)}
= \wt{z}_i(k) 
\sum_{p=0}^{[\frac{k-2}{2}]} (-1)^p M^{(k-1)}_p
- F_i(k-1 \bmod m) 
\sum_{p=0}^{[\frac{k-3}{2}]} (-1)^p M^{(k-2)}_p.
\end{align}
By comparing the degree $(k-1-2p)$-parts 
of \eqref{eq:ind-CandM} and \eqref{eq:g1-C-z},
we are reduced to showing
\begin{align*}
&M_0^{(k)}= \wt{z}_i(k) M^{(k-1)}_0,
\\
&M_p^{(k)}= \wt{z}_i(k) M^{(k-1)}_p - F_i(k-1 \bmod m) M^{(k-2)}_{p-1}
\quad \text{for} ~~p=1, \dots, [\frac{k-1}{2}].
\end{align*}
The first equality follows  from \eqref{eq:easyfromdef}.
To show the second, 
we suppose $1 \le p \le [\frac{k-1}{2}]$
and compute using \eqref{eq:A_1-L2}:
\begin{align*}
M^{(k)}_p 
=&
\sum_{\substack{\sigma \in \cM^{(k)}_p \\k-1 \not\in \sigma}} 
\prod_{j \in \sigma} F_i(j \bmod m) \prod_{\ol{j} \in \overline{\sigma}} \wt{z}_i(\ol{j})
+
\sum_{\substack{\sigma \in \cM^{(k)}_p \\k-1 \in \sigma}}
\prod_{j \in \sigma} F_i(j \bmod m) \prod_{\ol{j} \in \overline{\sigma}} \wt{z}_i(\ol{j})
\\
=&~
\wt{z}_i(k)
\sum_{\sigma \in \cM^{(k-1)}_p} 
\prod_{j \in \sigma} F_i(j \bmod m) \prod_{\ol{j} \in \overline{\sigma}} \wt{z}_i(\ol{j})
\\
&~+
F_i(k-1 \bmod m)
\sum_{\sigma \in \cM^{(k-2)}_{p-1}}
\prod_{j \in \sigma} F_i(j \bmod m) \prod_{\ol{j} \in \overline{\sigma}} \wt{z}_i(\ol{j})
\\
=&~ \wt{z}_i(k) M^{(k-1)}_p - F_i(k-1 \bmod m) M^{(k-2)}_{p-1}.
\end{align*}
We are done.\end{proof}

\begin{prop}
We have a formula 
\begin{align}\label{eq:A_1-Ainz}
  A^{(m)} = \sum_{p=0}^{[\frac{m}{2}]} (-1)^p \,T^{(m)}_{p} 
\quad \text {in } \cZ_{\cF},
\end{align}
where 
\begin{align}
  &\mathcal{T}_p^{(m)} = \bigl\{ \sigma  \subset \Z/m\Z;
~|\sigma|=p, ~
j \not= j'+1 ~\text{for any} ~j, j' \in \sigma \}, 
  \\
  &T^{(m)}_{p} = \sum_{\sigma \in \mathcal{T}^{(m)}_p} \prod_{j \in {\sigma}} F_i(j)\prod_{j' \in \overline{\sigma}} z_i(j').
\end{align}
Here, for $\sigma \in \mathcal{T}_k^{(m)}$ we set 
$\overline{\sigma} = \{ j \in (\Z / m\Z); ~j, j-1 \not\in \sigma \}$.
\end{prop}

\begin{proof}
From \eqref{eq:g1-A} we have 
\begin{align}\label{eq:A_1-AC}
\begin{split}
  \wt{A}^{(m)} &= \wt{z}_i(1) \sum_{p=0}^{[\frac{m-1}{2}]} (-1)^p M^{(m)}_{p} - F_i(m \bmod m) \sum_{p=0}^{[\frac{m-2}{2}]} (-1)^p M^{(m-1)}_{p} 
  \\
  & \qquad - F_i(1 \bmod m) \sum_{p=0}^{[\frac{m-2}{2}]} (-1)^p \tau_{\infty} (M^{(m-1)}_{p}).
\end{split}
\end{align}
By taking the degree $(m-2p)$-part and taking the image by $\pi$ of \eqref{eq:A_1-AC}, \eqref{eq:A_1-Ainz} reduces to
\begin{align}\label{eq:T-M}
T^{(m)}_p = 
\begin{cases}
z_i(1) \pi(M^{(m)}_0) & p = 0,
\\
z_i(1) \pi(M^{(m)}_{p}) + F_i(m) \pi(M^{(m-1)}_{p-1}) + F_i(1) \pi \circ \tau_{\infty} (M^{(m-1)}_{p-1}) & 1 \leq p \leq [\frac{m-1}{2}].
\end{cases}
\end{align}
The elements of $\mathcal{M}_p^{(m)}$ are subsets of $\{2,3,\ldots,m-1\}$, and we safely divert $\mathcal{M}_p^{(m)}$ to the set of subsets of $\{2,3,\ldots,m-1\} \subset \Z/m \Z$.
When $p=0$, by using \eqref{eq:easyfromdef}, the r.h.s. of \eqref{eq:T-M} coincides with $T_0^{(m)}$ as follows
$$
  z_i(1) \pi(M^{(m)}_0) = z_i(1) \prod_{p=2}^m z_i(p) = T_0^{(m)}. 
$$ 
When $1 \leq p \leq [\frac{m-1}{2}]$, the r.h.s. of \eqref{eq:T-M} is written as
\begin{align*}
z_i(1) & \sum_{\sigma \in \mathcal{M}^{(m)}_p} \prod_{j \in {\sigma}} F_i(j)\prod_{j' \in \overline{\sigma}} z(j') 
+ F_i(m) \sum_{\sigma \in \mathcal{M}^{(m-1)}_{p-1}} \prod_{j \in {\sigma}} F_i(j)\prod_{j' \in \overline{\sigma}} z(j')
\\
& \qquad \qquad + F_i(1) \sum_{\sigma \in \mathcal{M}^{(m-1)}_{p-1}} \prod_{j \in {\sigma}} F_i(j+1)\prod_{j' \in \overline{\sigma}} z(j'+1)
\\
&= \sum_{\substack{\sigma \in \mathcal{T}^{(m)}_p \\ 1,m \not\in \sigma}} \prod_{j \in {\sigma}} F_i(j)\prod_{j' \in \overline{\sigma}} z(j') 
+ \sum_{\substack{\sigma \in \mathcal{T}^{(m)}_{p} \\ m \in \sigma}} \prod_{j \in {\sigma}} F_i(j)\prod_{j' \in \overline{\sigma}} z(j')
\\
& \qquad \qquad
+ \sum_{\substack{\sigma \in \mathcal{T}^{(m)}_{p} \\ 1 \in \sigma}} \prod_{j \in {\sigma}} F_i(j+1)\prod_{j' \in \overline{\sigma}} z(j'+1).
\end{align*}
The last formula is nothing but $T^{(m)}_p$, since $\sigma \in \mathcal{T}^{(m)}_p$ does not contain $1$ and $m$ at the same time. Consequently, we obtain \eqref{eq:A_1-Ainz}.
\end{proof}

%%%%%%%%%%%%%%%%%%%%%%%%%
\section{Invariant subfield $\mathcal{Y}_m^W$: non-simply-laced cases}\label{sect:non-sl}
%%%%%%%%%%%%%%%%%%%%%%%%%

\subsection{Statements of the results}
When $\mg$ is associated to a non-simply-laced Dynkin diagram, 
we have $d' \in \{ 2d, 3d \}$ and 
$d_i \in \{ d, d' \}$ for any $i \in I$
as in \eqref{table:dd}. 
For $i \in I$ and 
\[
s \in 
{\Sigma}_i := \{s \in \Z; 1 \le s \le \frac{d_i}{d} \},
\] 
we define 
\begin{align}
\notag
&N_{i,s} := (d_i \Z + (s-1)d)/d'm \Z \subset d \Z/ d'm \Z,
\\
\label{eq:def-bldYF-delta}
&\boldsymbol{y}_{i,s} := \displaystyle{\prod_{n \in N_{i,s}} y_i(n)}, \quad 
\boldsymbol{F}_{i,s} := \displaystyle{\prod_{n \in N_{i,s}} F_i(n)}, \quad 
\delta_{i,s} = \boldsymbol{y}_{i,s} - \frac{\boldsymbol{F}_{i,s}}{\boldsymbol{y}_{i,s}}
\quad \in \mathcal{Y}_m.
\end{align}
Note that we have
${\Sigma}_i = \{ 1\}$ and $N_{i,1} = d \Z/ d'm \Z$
precisely when $d_i = d$. 
If this is not the case (i.e. $d_i=d'$),
we have
$|{\Sigma}_i| =d_i/d \in \{ 2, 3 \}$
and
$|N_{i, s}|=m$ for any $s \in \Sigma_i$.
Let us define a subfield $\mathcal{Z}_m^{(i) \prime}$
of $\mathcal{Y}_m$ as follows:
\begin{equation}\label{eq:def-sZprime}
\mathcal{Z}_m^{(i) \prime}
:=\begin{cases}
  \mathcal{Z}_m^{(i)} & \text{ if } \frac{d_i}{d} = 1,
  \\
  \mathcal{Z}_m^{(i)}(\delta_{i,1} \delta_{i,2}) & \text{ if } \frac{d_i}{d} = 2,  
  \\
  \mathcal{Z}_m^{(i)}(\delta_{i,1} \delta_{i,2}, \delta_{i,2} \delta_{i,3}) & \text{ if } \frac{d_i}{d} = 3.
\end{cases}
\end{equation}

The goal of this section is the following theorem.

\begin{thm}\label{thm:non-simply-laced}
Suppose that $\mg$ has a non-simply laced Dynkin diagram (that is, $\mg = B_\ell, C_\ell, F_4$ or $G_2$). Then, for $i \in I$  we have 
$\mathcal{Y}_m^{r_i} = \mathcal{Z}_m^{(i) \prime}$.
\end{thm}

In the rest of this section, we assume that $\mg$ is associated to a non-simply laced Dynkin diagram, and fix $i \in I$ and $m>1$.
We are going to prove a finer result than 
Theorem \ref{thm:non-simply-laced}
in Theorem \ref{thm:refined} below.
In order to formulate it, 
we need more notations.
For $s \in {\Sigma}_i$, define an automorphism $r_{i,s}$ of $\mathcal{Y}_m$ by
\begin{align}\label{eq:Weyl-r_is} 
  r_{i,s}(y_j(n)) = 
\begin{cases}
\displaystyle{\frac{P_i(n-2d_i)}{P_i(n-d_i)} y_i(n) X_i(n-d_i)} & j = i, n \in N_{i,s},
\\[1mm]
y_j(n) & otherwise,
\end{cases}
\end{align}
where $(j, n) \in I \times d\Z/d'm\Z$.
From \eqref{a-y} and \eqref{eq:def-fsmall} 
we get
\begin{equation}\label{eq:PXfixed-ris}
r_{i, s}(P_i(n))=P_i(n)
~\text{and}~
r_{i, s}(X_i(n))=X_i(n)
\quad \text{if}~ n \not\in N_{i, s}.
\end{equation}

\begin{lem}\label{lem:ri-Yf-non}
For any $s, s' \in {\Sigma}_i$ satisfying $s \neq s'$, 
the following hold.
\begin{itemize}
\item[(1)]
The actions of $r_{i,s}$ and $r_{i,s'}$ on $\cY_m$ are commutative. 
\item[(2)]
We have $\displaystyle{r_i = \prod_{s \in {\Sigma}_i} r_{i,s}}$. 
\item[(3)]
We have 
$r_{i,s}(\delta_{i,s'})=\delta_{i,s'}$.
\end{itemize}
\end{lem}

\begin{proof}
(1) 
Let $(j,n) \in I \times d \Z / d'm \Z$.
If $j \not=i$ or if $n \not\in (N_{i, s} \cup N_{i, s'})$,
then
it follows that $r_{i,s} r_{i,s'}(y_j(n)) = y_j(n)= r_{i,s'} r_{i,s}(y_j(n))$ from \eqref{eq:Weyl-r_is}. 
Otherwise, when $n \in N_{i,s}$ we have $r_{i,s} r_{i,s'}(y_i(n)) = r_{i,s}(y_i(n)) = r_{i,s'} r_{i,s}(y_i(n))$, and when $n \in N_{i,s'}$ we have $r_{i,s} r_{i,s'}(y_i(n)) = r_{i,s'}(y_i(n)) = r_{i,s'} r_{i,s}(y_i(n))$, from \eqref{eq:Weyl-r_is} and \eqref{eq:PXfixed-ris}.
Thus the claim follows.

(2) Due to (1), this is nothing but a paraphrase of the definition of $r_i$ as a composition of commuting operators $r_{i,s}$ for $s \in {\Sigma}_i$.

(3)
This follows from definitions
\eqref{eq:def-bldYF-delta} and \eqref{eq:Weyl-r_is}.
\end{proof}

\begin{prop}\label{lem:Weyl-prefolding}
For $s \in {\Sigma}_i$, the order of $r_{i,s}$ is two.
\end{prop}

We postpone the proof of this proposition 
to \S \ref{app:proof-order2}.
When $d_i = d$, $r_{i,1}$ coincides with $r_i$ \eqref{Weyl-y}, 
thus the order of $r_{i,s}$ is two.
When $d_i \neq d$, Proposition \ref{lem:Weyl-prefolding} can be proved in the same way as \cite{I21}, by applying cluster mutations. 
Our proof in \S \ref{app:proof-order2}
does not use cluster mutation.

Let $R_i$ be the subgroup of
automorphisms of $\cY_m$ generated by
$r_{i, s}$ for all $s \in \Sigma_i$.
By Lemma \ref{lem:ri-Yf-non} and Proposition \ref{lem:Weyl-prefolding},
we have an isomorphism
\[
(\Z/2\Z)^{\Sigma_i} \overset{\cong}{\longrightarrow} R_i;
\qquad
(\epsilon_s)_{s \in \Sigma_i} 
\mapsto \prod_s r_{i, s}^{\epsilon_i}.
\]
The following refines
Theorem \ref{thm:non-simply-laced},
whose proof will be completed in \S \ref{ss:pf4}.

\begin{thm}\label{thm:refined}
The $R_i$-invariant subfield $\cY_m^{R_i}$ of $\cY_m$
agrees with $\cZ_m^{(i)}$,
hence the extension 
$\cY_m/\cZ_m^{(i)}$ is Galois with group $R_i$.
Moreover, we have
\begin{equation}\label{eq:refined}
\cY_m=\cZ_m^{(i)}(\delta_{i, s}; ~s \in \Sigma_i),
\quad
\delta_{i, s}^2 \in \cZ_m^{(i)},
\quad
r_{i, s}(\delta_{i, s})=-\delta_{i, s}
\quad
\text{for any}~ s \in \Sigma_i.
\end{equation}
\end{thm}

\subsection{First reduction}
To prove the results in the previous subsection,
we employ a similar idea as \eqref{eq:def-cF}.
Let us define subfields of $\mathcal{Y}_m$ as follows:

\begin{equation}
\label{eq:def-cF-non}
\begin{aligned}
&\mathcal{Z}_m^{(i,s)} := 
\C\bigl(z_i(k), y_j(n);~ 
k \in N_{i,s}, ~
(j, n) \in I \times d\Z/d'm\Z,~
j \not= i ~\text{or}~n \not\in N_{i, s}\bigr),
%
%&\mathcal{Z}_m^{(i,s)} = \C\bigl(z_i(k), y_i(n),y_j(n'); k \in N_{i,s}, \displaystyle{n \in \bigcup_{s' \in {\Sigma}_i \setminus \{s\}}} N_{i,s'}, j \in I \setminus \{i\}, n' \in d\Z/d'm\Z \bigr).
\\
&\cF := \C(F_i(n); ~n \in d \Z/d'm\Z), \\
&\cY_{\cF,s} := \cF(y_i(n); ~n \in N_{i,s}), \\
&\cZ_{\cF,s} := \cF(z_i(n); ~n \in N_{i,s}).
\end{aligned}
\end{equation}
Note that for all $n \in N_{i,s}$ we have
\begin{equation}\label{eq:Xi-cYcF-non}
X_i(n), ~P_i(n), ~\delta_{i, s} \in \cY_{\cF,s}
\end{equation}
by \eqref{a-y} and \eqref{eq:def-fsmall}.

\begin{lem}\label{lem:ri-Yf-non2}
The restriction of $r_{i,s}$ to $\cZ_{\cF,s}$ is the identity,
and
we have $r_{i,s}(\cY_{\cF,s}) \subset \cY_{\cF,s}$.
\end{lem}
\begin{proof}
This is proved in the same way as Lemma \ref{lem:ri-Yf}, by using \eqref{eq:a-y-g}, \eqref{eq:def-fsmall}, \eqref{eq:Weyl-r_is} and \eqref{eq:Xi-cYcF-non}.
\end{proof}

In a similar way as \eqref{eq:diag-fld} 
the relations of the fields is summarized in a diagram: 
\begin{equation}\label{eq:diag-fld-non}
\xymatrix{
\cZ_m^{(i)} \ar@{}[r]|*{\subset} 
\ar@{}[dr]|{\rotatebox{45}{$\cup$}} &
\cZ_m^{(i,s)} \ar@{}[r]|*{\subset} \ar@{}[d]|{\bigcup} &
\cY_m^{r_{i,s}} \ar@{}[r]|*{\subset} \ar@{}[d]|{\bigcup} &
\cY_m  \ar@{}[d]|{\bigcup} 
\\
&
\cZ_{\cF,s} \ar@{}[r]|*{\subset} & %\ar@{_{(}->}[ul] &
\cY_{\cF,s}^{r_{i,s}} \ar@{}[r]|*{\subset}  &
\cY_{\cF,s},
}
\end{equation}
where $\mathcal{Y}_{\cF,s}^{r_{i, s}}$ is
the $r_{i,s}$-invariant subfield of $\cY_{\cF,s}$.
The following is an analogue of Lemma \ref{lem:reduction}.

\begin{lem}\label{lem:reduction2}
The assertions 
\begin{equation}\label{eq:refined2}
\cY_{\cF, s}=\cZ_{\cF, s}(\delta_{i, s}),
\quad
\delta_{i, s}^2 \in \cZ_{\cF, s},
\quad
r_{i, s}(\delta_{i, s})=-\delta_{i, s}
\quad 
\text{for any}~ s \in \Sigma_i
\end{equation}
imply Theorems 
\ref{thm:non-simply-laced} and \ref{thm:refined}.
\end{lem}
\begin{proof}
Since $\cY_m$ is 
the composition field of 
$\cY_{\cF, s}$ and $\cZ_m^{(i, s)}$,
the same argument as Lemma \ref{lem:reduction}
shows that \eqref{eq:refined2} implies
$\cZ_m^{(i, s)}=\cY_m^{r_{i, s}}$.
It then follows that
\[ \cY_m^{R_i} 
= \bigcap_{s \in \Sigma_i} \cY_m^{r_{i, s}}
= \bigcap_{s \in \Sigma_i} \cZ_m^{(i, s)}
\supset \cZ_m^{(i)},
\]
and hence 
$[\cY_m : \cZ_m^{(i)}] \geq |R_i|=2^{d_i/d}$.
On the other hand, 
$\cY_m$ is also the composition field of 
$\cZ_m^{(i)}$ and $\cY_{\cF, s}$ 
where $s$ ranges over $\Sigma_i$.
Thus \eqref{eq:refined2} implies \eqref{eq:refined}.
In particular this shows that
$[\cY_m : \cZ_m^{(i)}] \le 2^{d_i/d}$,
whence $\cY_m^{R_i}=\cZ_m^{(i)}$.
We have proved Theorem \ref{thm:refined}.
Theorem \ref{thm:non-simply-laced}
then follows Lemma \ref{lem:ri-Yf-non} and
Proposition \ref{lem:Weyl-prefolding}.
\end{proof}

\subsection{The proof}\label{ss:pf4}
In order to prove \eqref{eq:refined2},
we introduce the Laurent polynomial rings 
\[
\mathcal{Y}_{\infty} := \cF[\wt{y}_i(n)^{\pm 1}; n \in d\Z]
~ \supset ~ 
\mathcal{Y}_{\infty,s} := \cF[\wt{y}_i(n)^{\pm 1}; n \in \wt{N}_{i,s}]
\qquad (s \in \Sigma_i)
\]
on the set of commuting variables
$\wt{y}_i(n)$ over $\cF$,
where we put $\wt{N}_{i,s} := d_i \Z + (s-1)d \subset d\Z$.
We also introduce its $\cF$-subalgebra
\begin{equation}\label{psi-def3-s}
\mathcal{Z}_{\infty,s}:= \cF[\wt{z}_i(n);~ n \in \wt{N}_{i,s}] \subset \cY_{\infty,s},
\qquad
  \wt{z}_i(n) = \wt{y}_i(n) + \frac{F_i(n \bmod d'm)}{\wt{y}_i(n+d_i)} .
\end{equation}
One checks 
that the set $\{ \wt{z}_i(n) ;~ n \in d\Z \}$ 
is algebraically independent over $\cF$
and thus
$\cZ_{\infty,s}$ is a polynomial ring over $\cF$,
as in Lemma \ref{lem:alg-indep}.
We generalize Definition \ref{lem:shift-op} as follows.
\begin{definition}\label{lem:shift-op-non}
\begin{itemize}
\item[(1)] 
We define a $\C$-algebra automorphism $\tau_{m} : \cY_{m} \to \cY_{m}$ given by $\tau_m y_j(n) = y_j(n+d)$ for any $(j,n) \in I \times d \Z / d'm \Z$. 
We remark that $\tau_m$ restricts to an isomorphism $\cY_{m,s} \cong \cY_{m,s+1}$ for 
%($i \in I$ and)
$s \in \Sigma_i$, where $s+1$ is understood as $1$ if $s=d_i/d$. 
Thus a composition $\tau_m^{\frac{d_i}{d}}$ yields an automorphism of $\cY_{m,s}$ for $s \in \Sigma_i$.
We also remark that
$\tau_m$ restricts to an automorphism $\tau_\cF$ of $\cF$.
\item[(2)]
We define a $\C$-algebra automorphism $\tau_{\infty} : \cY_{\infty} \to \cY_{\infty}$ by $\tau_m \wt{y}_i(n) = \wt{y}_i(n+d)$ for any $n \in d \Z$
and $\tau_\infty|_{\cF}=\tau_\cF$.
Similarly to (1), $\tau_\infty^{\frac{d_i}{d}}$ restricts to 
an automorphism of 
$\cY_{\infty,s}$ for $s \in \Sigma_i$.
\item[(3)]
Let $\pi : \cY_{\infty, s} \to \cY_{\cF,s}$ be an $\cF$-algebra homomorphism characterized by
$\pi \wt{y}_i(n)=y_i(n \bmod d' m)$
for all $n \in \wt{N}_{i, s}$. 
It holds that 
$\tau_m^{\frac{d_i}{d}} \circ \pi 
= \pi \circ \tau_\infty^{\frac{d_i}{d}}$.
\end{itemize}
\end{definition}

In the rest of this section, besides $i \in I$ and $m > 1$ we fix $s \in \Sigma_i$. We define polynomials $\wt{A}^{(k)}, \wt{C}^{(k)}$ in $\cZ_{\infty,s}$ as follows. First we define $\wt{C}^{(k)}$ for $k \in \Z_{>0}$ by
\begin{align}
  \label{eq:g1-C12-non}
  &\wt{C}^{(1)} = 1, ~~ \wt{C}^{(2)} = \wt{z}_i(2d_i+(s-1)d),
  \\
  \label{eq:g1-Ck-non}
  \begin{split}
  &\wt{C}^{(k)} = \wt{z}_i(kd_i+(s-1)d)\, \wt{C}^{(k-1)} 
  \\
  & \qquad \quad 
  - F_i((k-1)d_i+(s-1)d \bmod d'm)\, \wt{C}^{(k-2)} \quad (k \ge 3),
  \end{split}
\end{align} 
and next, for $k \geq 2$ we define $\wt{A}^{(k)}$ as
\begin{align}
\label{eq:g1-A-non}
  \begin{split}
  &\wt{A}^{(k)} = \wt{z}_i(d_i+(s-1)d) \wt{C}^{(k)} - F_i(kd_i+(s-1)d \bmod d'm) \,\wt{C}^{(k-1)} 
  \\ 
  & \qquad \quad - \tau_{\infty}^{\frac{d_i}{d}} (F_i(kd_i+(s-1)d \bmod d'm)\, \wt{C}^{(k-1)}).
  \end{split} 
%\end{aligned}
\end{align} 
Further we define elements of $\cY_{\infty, s}$ 
for $k \geq 2$ and $n \in N_{i, s}$ by 
\begin{equation}%\label{eq:g1-def-small-f2}
\wt{D}_n^{(k)} = 
1 + \sum_{p=0}^{k-2} \wt{X}_i(n)\wt{X}_i(n-d_i)\cdots \wt{X}_i(n-p d_i),
\quad
\wt{X}_i(n) = \frac{F_i(n \bmod d'm)}{\wt{y}_i(n) \wt{y}_i(n+d_i)}.
\end{equation} 
For $n \in d \Z$ and $2 \leq k \leq d'm/d_i$, we define
\begin{align}\label{eq:g1-ACD-non}
A^{(k)}:=\pi(\wt{A}^{(k)}) \in \cZ_{\cF,s},~
\quad
C^{(k)}:=\pi(\wt{C}^{(k)}) \in \cZ_{\cF,s},
\quad
D_n^{(k)}:=\pi(\wt{D}_n^{(k)}) \in \cY_{\cF,s}.
\end{align}

Now the following lemma and theorem are proved 
in the precisely same manner as 
Lemma \ref{lem:g1-C-f} and Theorem \ref{thm:g1-delta}.
We omit the details.

\begin{lem}\label{lem:g1-C-f-non}
For $s \in \Sigma_i$ it is satisfied that 
$$
\wt{C}^{(k)} 
= \wt{D}_{k d_i +(s-1)d}^{(k)}\, \wt{y}_i(2 d_i +(s-1)d) \wt{y}_i(3 d_i +(s-1)d)\cdots \wt{y}_i(k d_i +(s-1)d)
$$ 
in $\cY_{\infty,s}$. 
\end{lem}

\begin{thm}\label{thm:g2-delta}
\begin{enumerate}
\item 
We have $\displaystyle{A^{(d'm/d_i)} = \boldsymbol{y}_{i,s} + \frac{\boldsymbol{F}_{i,s}}{\boldsymbol{y}_{i,s}}}$ in $\cY_{\cF,s}$. 
\item 
We have $2 (y_i(1) C^{(d'm/d_i)}- F_i(d_i m \mod d'm) C^{((d'm/d_i)-1)})- A^{(d'm/d_i)}=\delta_{i,s}$.
%We have $2 (y_i(1) C^{(d'm/d_i)}- F_i(\overline{d}_i m) C^{(d'm/d_i-1)})- A^{(d'm/d_i)}=\delta_{i,s}$ in $\cY_{\cF,s}$.
In particular, this element is invariant under $\tau_m^{\frac{d_i}{d}}$. 
\item 
We have 
$\cY_{\cF,s}=\cZ_{\cF, s}(\delta_{i,s})$,~
$(\delta_{i,s})^2 
=(A^{(d'm/d_i)})^2-4 \boldsymbol{F}_{i,s} \in \cZ_{\cF,s}$
and 
$r_{i,s}(\delta_{i,s})=-\delta_{i,s}$.
\end{enumerate}
\end{thm}

\begin{proof}[Proof of 
Theorems \ref{thm:non-simply-laced} and \ref{thm:refined}]
By Lemma \ref{lem:reduction2},
it suffices to prove \eqref{eq:refined2},
which is nothing but Theorem \ref{thm:g2-delta} (3).
\end{proof}

%%%%%%%%%%%%%%%%%%%
\subsection{Proof of Proposition \ref{lem:Weyl-prefolding}}
\label{app:proof-order2}
%%%%%%%%%%%%%%%%%%%
We introduce a key lemma to prove Proposition \ref{lem:Weyl-prefolding}.

\begin{prop}\label{prop:r_is_on_f} 
For $i \in I$, $s \in {\Sigma}_i$ and $n \in N_{i,s}$ we have 
\begin{align}\label{eq:f-prop}
  r_{i,s}(P_i(n)) = P_i(n-d_i) \frac{X_i(n)}{\displaystyle{\prod_{k \in d_i \Z / d'm \Z} X_i(n + k)}} \quad \text{ in } \cY_{\cF,s}.
\end{align}
\end{prop}

We reduce the proposition to the case $s=1$ 
by using $\tau_m$ from Definition \ref{lem:shift-op-non} (1),
which verifies $r_{i, s+1} \circ \tau_m = \tau_m \circ r_{i, s}$.
Note that in this case we have $N_{i,1} = d_i \Z / d'm \Z$.
To ease the notations we write $X_n$ and $P_n$ for $X_i(n d_i)$ and $P_i(n d_i)$ in $\cY_{\cF,1}$ respectively.

Note that for $n \in N_{i,1}$ we have $X_n P_{n-1}-P_n = \prod_{k \in N_{i,1}} X_k - 1$ which does not depend on $n$.
In particular, for any $n, n' \in N_{i,1}$ it holds that
\begin{align}\label{eq:f-f}
  X_n P_{n-1}+ P_{n'} = X_{n'} P_{n'-1} + P_n \quad \text{ in } \cY_{\cF,1}. 
\end{align}
Also note that for $n \in N_{i,1}$ and $2 \leq k \leq d'm/d_i$,  
$D_n^{(k)}$ defined at \eqref{eq:g1-ACD-non} is now written as  
\begin{equation}\label{eq:g-def-small-f}
D_n^{(k)} = 1 + \sum_{p=0}^{k-2} X_n X_{n-1}\cdots X_{n-p}.
\end{equation} 

\begin{lem}
We have the following formula:
\begin{align}\label{eq:D-X-f}
&D_n^{(k)} X_{n-k} P_{n-k-1}+P_{n}=D_n^{(k+1)} P_{n-k} \quad \text{ for } 2 \leq k \leq \frac{d'm}{d_i} - 1. 
\end{align}
\end{lem}

\begin{proof}
The sum of 
$\sum_{p=0}^{k-2} X_n X_{n-1} \cdots X_{n-p} P_{n-p-1}$
and the l.h.s. of \eqref{eq:D-X-f} is calculated as follows:
\begin{align*}
  &D_n^{(k)} X_{n-k} P_{n-k-1}+P_{n} + \sum_{p=0}^{k-2} X_n X_{n-1} \cdots X_{n-p} P_{n-p-1}
\\
& \quad = \underline{X_{n-k} P_{n-k-1}+P_{n}} + \sum_{p=0}^{k-2} X_n X_{n-1} \cdots X_{n-p} (\underline{X_{n-k} P_{n-k-1} + P_{n-p-1}}) \quad \text{(from \eqref{eq:g-def-small-f})}
\\
& \quad = X_n P_{n-1} + P_{n-k} + \sum_{p=0}^{k-2} X_n X_{n-1} \cdots X_{n-p} (X_{n-p-1} P_{n-p-2} + P_{n-k}) 
\\
& \hspace{7cm} 
\text{(apply \eqref{eq:f-f} to the underlined parts)}
\\
& \quad = \left(1 + \sum_{p=0}^{k-2} X_n X_{n-1} \cdots X_{n-p} \right) P_{n-k}
+ X_n X_{n-1} \cdots X_{n-k+1} P_{n-k} 
\\
& \qquad \qquad + X_n P_{n-1} + \sum_{p=0}^{k-3} X_n X_{n-1} \cdots X_{n-p} X_{n-p-1} P_{n-p-2}
\\
& \quad = D^{(k+1)}_n P_{n-k}
+ \sum_{p=0}^{k-2} X_n X_{n-1} \cdots X_{n-p} P_{n-p-1} \quad \text{(from \eqref{eq:g-def-small-f})}. 
\end{align*}
Hence we obtain \eqref{eq:D-X-f}. 
\end{proof}

\begin{lem}
We have the following:
\begin{align}\label{eq:r-f-proof}
  r_{i, 1}(D_n^{(k)}) = \frac{P_{n-1}}{X_{n-1}X_{n-2} \cdots X_{n-k+1} P_{n-k}} D_n^{(k)} \quad \text{ for } 2 \leq k \leq d'm/d_i.
\end{align}
\end{lem}

\begin{proof}
We prove this by induction on $k$. Note that from \eqref{eq:Weyl-r_is} $r_{i,1}$ acts on $X_n$ as
\begin{align}\label{eq:r-on-X}
  r_{i,1}(X_n) = \frac{P_n}{X_{n-1} P_{n-2}}.
\end{align}
When $k=2$, we have 
$$ 
  r_{i,1}(D_n^{(2)}) = r_{i,1}(1+X_{n}) = 1 + \frac{P_{n}}{X_{n-1}P_{n-2}} =  \frac{P_{n-1}}{X_{n-1}P_{n-2}} D_n^{(2)}
$$
where we use \eqref{eq:f-f} with $n'=n-1$ at the last equality.
By induction hypothesis and \eqref{eq:r-on-X}, for $k \geq 2$ we have 
\begin{align*}
  r_{i,1}(D_n^{(k+1)}) &= r_{i,1}(D_n^{(k)}) + r_{i,1}(X_{n}X_{n-1}\cdots X_{n-k+1})
  \\
  &= \frac{P_{n-1}}{X_{n-1}X_{n-2} \cdots X_{n-k+1} P_{n-k}} D_n^{(k)} 
     + \frac{P_{n}P_{n-1}}{X_{n-1}X_{n-2} \cdots X_{n-k} P_{n-k}  P_{n-k-1}}
  \\
  & \hspace{5cm} \text{(from the assumption and \eqref{eq:r-on-X})} 
  \\
  &= \frac{P_{n-1}}{X_{n-1}X_{n-2} \cdots X_{n-k} P_{n-k} P_{n-k-1}} \left(D_n^{(k)} X_{n-k} P_{n-k-1}+P_{n} \right)
  \\
  &= \frac{P_{n-1}}{X_{n-1}X_{n-2} \cdots X_{n-k} P_{n-k-1}} D_n^{(k+1)}
     \text{(from \eqref{eq:D-X-f})}.
\end{align*}
\end{proof}

\begin{proof}[Proof of Proposition \ref{prop:r_is_on_f}]
We obtain \eqref{eq:f-prop} from \eqref{eq:r-f-proof} by setting $k=d'm/d_i$, 
due to the fact $D_n^{(d'm/d_i)} = P_n$.
\end{proof}

Now we are ready to prove Proposition \ref{lem:Weyl-prefolding}.
By the definition of $r_{i,s}$ \eqref{eq:Weyl-r_is}, it is suffice to show
$r_{i,s}^2(y_i(n)) = y_i(n)$ for $n \in N_{n,s}$. By using \eqref{eq:f-prop}, we obtain 
\begin{align*}
r_{i,s}^2&(y_i(n))
=
r_{i,s}\left(\frac{P_i(n-2d_i)}{P_i(n-d_i)} \frac{F_i(n-d_i)}{y_i(n-d_i)}\right)
\\
&=\frac{P_i(n-3d_i) X_i(n-2d_i)}{P_i(n-2d_i) X_i(n-d_i)}F_i(n-d_i) \frac{P_i(n-2d_i)}{P_i(n-3d_i) y_i(n-d_i) X_i(n-2d_i)}   
= y_i(n).
%
%
%  y_i(n) &\stackrel{r_{i,s}}{\mapsto} \frac{P_i(n-2d_i)}{P_i(n-d_i)} \frac{F_i(n-d_i)}{y_i(n-d_i)}
%\\
%&\stackrel{r_{i,s}}{\mapsto} \frac{P_i(n-3d_i) X_i(n-2d_i)}{P_i(n-2d_i) X_i(n-d_i)}F_i(n-d_i) \frac{P_i(n-2d_i)}{P_i(n-3d_i) y_i(n-d_i) X_i(n-2d_i)}   
%= y_i(n).
\end{align*}
This completes the proof.

%%%%%%%%%%%%%%%%%%%%%%%%%%%

\end{document}